\documentclass[a4paper]{article}

\usepackage[left=2cm,top=3cm,bottom=3cm,right=2cm]{geometry}
\usepackage{amsmath, amssymb, amsfonts, amsthm, amscd, epsfig, multicol,enumerate,float}
\usepackage{mathpazo} 
\usepackage[scaled=.95]{helvet} 

\usepackage{courier} 
\usepackage{fancyheadings}
\theoremstyle{definition}
\usepackage{soul}
\usepackage[usenames]{color}
\theoremstyle{definition}
\newcounter{dummy} \numberwithin{dummy}{section}
	\newtheorem{theorem}[dummy]{Theorem}

	\newtheorem{remark}[dummy]{Remark}
	\newtheorem{corollary}[dummy]{Corollary}

\begin{document}
\title{Identities for Catalan's constant arising from integrals depending on a parameter}

\author{Federica Ferretti\thanks{Free University of Bozen\textcolor{blue}{\tt\, Federica.Ferretti@unibz.it}}\,\,Alessandro Gambini\thanks{Sapienza Universit\`{a} di Roma\textcolor{blue}{\tt \,alessandro.gambini@uniroma1.it}}\,\,and\,Daniele Ritelli\thanks{Dipartimento di Scienze Statistiche Universit\`{a} di  Bologna \textcolor{blue}{\tt \,daniele.ritelli@unibo.it}}}
\date{}
\maketitle

\begin{abstract}
In this paper we provide some relationships between Catalan's constant and the $_3{\rm F}_2$ and $_4{\rm F}_3$ hypergeometric functions, deriving them from some parametric integrals. In particular, using the complete elliptic integral of the first kind, we found an alternative proof of a result of Ramanujan for $_3{\rm F}_2$, a second identity related to $_4{\rm F}_3$ and using the complete elliptic integral of the second kind we obtain an identity by Adamchik.

\vspace{.5cm}


\vspace{.5cm}




\end{abstract}

\section{Introduction}

Catalan's constant $G$ was defined by Eug\`{e}ne Charles Catalan, who introduced this constant in \cite[eq. (4) p. 23]{catalan0} as the alternating series
\begin{align*}\label{G0}
G=\sum_{n=0}^{\infty}\frac{(-1)^n}{(2n+1)^2}.
\end{align*}
It is well-known, \cite[eq. (1) p. 1]{catalan}, that we may also express this constant with the following definite integral
\[
\int_{0}^1\frac{\arctan x}{x}{\rm d}x.
\]
This integral has stimulated the interest of Ramanujan, \cite{integral}.
$G$ was also identified by James W.L. Glaisher in 1877, see \cite{Glaisher}. Its value is approximately $G\cong 0.915965594177\dots$ and actually it is not known if it is a rational number. 
The constant $G$ is somewhat ubiquitous since it appears in many 
occurrences connected to definite integrals or series summations; we give a (surely) non exhaustive list of papers that provide several interesting occurrences of $G$, including \cite{adamchik, alkan, bradley, interplay, choi, Nick, deriving} and \cite{yang}. There are also interesting connections between Catalan's constant and the Clausen's integrals, that is
\[
\mathrm{Cl}_2(\theta)= -\int_0^{\theta}\log\left(2 \, \sin\frac{t}{2}\right) \, {\rm d}t,
\]
the Hurwitz zeta function, 
Euler-Mascheroni constant and many other special functions, as one can see for instance in \cite{interplay}, \cite{choi}, \cite{yang}.
The starting point of this paper is the connection of $G$ with the complete elliptic integrals of the first
and second kind, namely
\[
G=\frac12\int_0^1K(s)\,{\rm d}s,\quad\text{and}\quad G=\int_0^1E(s)\,{\rm d}s-\frac12
\]
where the two complete elliptic integrals for $|s|\leq1$ are given by:
\[
K(s)=\int_0^1\frac{{\rm d}x}{\sqrt{(1-x^2)(1-s^2x^2)}},\quad 
E(s)=\int_0^1\sqrt{\frac{1-s^2x^2}{1-x^2}}\,{\rm d}x. 
\]
The basic idea of this paper is to arrive at well-known formulas using Catalan's constant through Feynman's favourite technique of differentation under the integral sign. We can express $G$ in terms of the hypergeometric function $_3{\rm F}_2$, providing in Section \ref{due} an 
alternative proof of the representation of $G$ in terms of $_3{\rm F}_2$, due to Ramanujan: in his second notebook Chapter 10, as reported by Berndt, 
\cite[p. 40 eq. (29.3)]{berndt}, the following relation is established:
\begin{equation}\label{prodigi}
\frac{4G}{\pi}=\,_{3}\mathrm{F}_{2}\left(  
\begin{array}{c}
1/2,1/2,1/2 \\[2mm]
1,3/2
\end{array}
;1\right).
\end{equation}
We deem interesting provide some details of the Ramanujan's way to formula (\ref{prodigi}), in order to emphasize how our approach requires a less sophisticated, and very different, mathematical machinery. Equation (\ref{prodigi}) in \cite[Entry 29(d) p. 40]{berndt} is obtained as a particular case of the following $_3{\rm F}_2$ identity (\ref{prodigii}) which holds true for $n>-\frac32$
\begin{equation}\label{prodigii}
\,_{3}\mathrm{F}_{2}\left(  
\begin{array}{c}
1/2,1,n+3/2 \\[2mm]
3/2,n+2
\end{array}
;1\right)=\sqrt{\pi}\,\frac{\Gamma(n+2)}{\Gamma(n+\frac32)}\,_{3}\mathrm{F}_{2}\left(  
\begin{array}{c}
1/2,1/2,-n \\[2mm]
1,3/2
\end{array}
;1\right).
\end{equation}
In fact taking the particular value $n=-\frac12$ equation (\ref{prodigii}) furnishes
\begin{equation}\label{prodigiii}
\,_{3}\mathrm{F}_{2}\left(  
\begin{array}{c}
1/2,1,1\\[2mm]
3/2,3/2
\end{array}
;1\right)=\frac{\pi}{2}\,_{3}\mathrm{F}_{2}\left(  
\begin{array}{c}
1/2,1/2,1/2 \\[2mm]
1,3/2
\end{array}
;1\right).
\end{equation}
The left-hand side of (\ref{prodigiii}) is evaluated in the first Ramanujan's notebook: in fact using the Whipple quadratic transformation type identity, see for instance \cite[p.190]{erdely}, given in \cite[Entry 32, Eq. (32.2)  p.288]{berndt1}, which we express using the $_3{\rm F}_2$ notation and recalling that (\ref{prodigiv}) holds true for $|x|\leq1$:
\begin{equation}\label{prodigiv}
\,_{3}\mathrm{F}_{2}\left(  
\begin{array}{c}
1/2,1,1\\[2mm]
3/2,3/2
\end{array}
;\frac{4x}{(1+x)^2}\right)=(1+x)\sum_{n=0}^{\infty}\frac{(-x)^n}{(2n+1)^2};
\end{equation}
taking $x=1$ in (\ref{prodigiv}) we see that the left-hand side of (\ref{prodigiii}) is exactly $2G$ which completes the illustration of the Ramanujan approach to (\ref{prodigi}).

Finally, regarding formula (\ref{prodigi}), it is also worth mentioning the contribution of \cite{sriva}.

In  Section \ref{due} we also present (always starting from a suitable parametric integral) two representations of $G$ that are interrelated with Euler's log-sine and log-cosine integrals, see \cite[Section 990]{rober}  or the most recent \cite[Section 2.4 pages 64-65]{nahin}:
\[
 \int_{0}^{\frac{\pi}{2}}\log\left(\sin x\right){\rm d}x=\int_{0}^{\frac{\pi}{2}}\log\left(\cos x\right){\rm d}x=-\frac{\pi}{2}\,\log 2.
\]
In Section \ref{tre}, we present a further hypergeometric connection of $G$ with the complete elliptic integral of second kind, using its $_2{\rm F}_1$ representation
\begin{equation}\label{repre}
E(k)=\frac{\pi}{2}\,_{2}\mathrm{F}_{1}\left(  
\begin{array}{c}
-1/2,1/2 \\[2mm]
1
\end{array}
;k\right).
\end{equation}

\section{Catalan's constant from complete elliptic integral of the first kind}\label{due}

We recall that $_{3}\mathrm{F}_{2}$ and $_{4}\mathrm{F}_{3}$ are the hypergeometric generalized function, defined for $|x|<1$ by the power series
\[
\begin{split}
_3\mathrm{F}_{2}\left( 
\begin{array}{c}
a_1,a_2,a_3 \\ 
b_1,b_2
\end{array}
; x\right)
&=\sum_{k=0}^{\infty}\frac{(a_1)_k\,(a_2)_k\,(a_3)_k}{(b_1)_k\,(b_2)_k}\frac{x^k}{k!},\\
 _4\mathrm{F}_{3}\left( 
\begin{array}{c}
a_1,a_2,a_3,a_4 \\ 
b_1,b_2,b_3
\end{array}
; x\right)
&=\sum_{k=0}^{\infty}\frac{(a_1)_k\,(a_2)_k\,(a_3)_k\,(a_4)_k}{(b_1)_k\,(b_2)_k\,(b_3)_k}\frac{x^k}{k!}
\end{split}
\]
where $(a)_k$ is the standard notation for the Pochhammer symbol (increasing factorial):
\[
(a)_k=\frac{\Gamma(a+k)}{\Gamma(a)}=a(a+1)\cdots(a+k-1).
\]
It should be remembered that $_{3}\mathrm{F}_{2}$ converges at $\,x=1\,$ whenever $\,{\rm Re}{\left(b_1+b_2-a_1-a_2-a_3\right)} > 0\,$ and $_{4}\mathrm{F}_{3}$ converges at $\,x=1\,$ whenever $\,{\rm Re}{\left(b_1+b_2+b_3-a_1-a_2-a_3-a_4\right)} > 0.$ See Section~2.2 page 45 of~\cite{slater}.

The following Theorem \ref{uno} provides an integration formula which is related to Lemma 1 of \cite{lima}, but here the result is obtained by means of elementary techniques. 
\begin{remark}
Formula \eqref{e1} as given in Theorem \ref{uno} below, is not found in the classical repertories, for instance in Sections 4.52-4.59 pp. 600-607 of \cite{gradshteyn} it does not appear, but since it can be evaluated using Mathematica$_\text{\textregistered}$ this means that it can be obtained using known formulas. Our elementary approach, which is based on the use of differentiation under the integral sign, gives us a rigorous proof of this identity.
\end{remark}
\begin{theorem}\label{uno}
The following formula for $|s|<1$ holds:
\begin{equation}\label{e1}
\int_0^1\frac{\arcsin (sx)}{x\sqrt{1-x^2}}{\rm d}x=\frac{\pi}{2}s\;_{3}\mathrm{F}_{2}\left(  
\begin{array}{c}
1/2,1/2,1/2 \\[2mm]
1,3/2
\end{array}
;s^2\right).
\end{equation}
\end{theorem}

\begin{proof}
Let
\[
A(s):=\int_0^1\frac{\arcsin(sx)}{x\sqrt{1-x^2}}{\rm d}x.
\]
Observe that, since
\[
\lim_{x\to0}\frac{\arcsin(sx)}{x}=s
\]
we have, being $|s|<1$:
\[
\left|\frac{\arcsin(sx)}{x\sqrt{1-x^2}}\right|\leq\frac{1}{\sqrt{1-x^2}}
\]
thus we can differentiate under the integral sign to reach
\[
A'(s)=\int_0^1\frac{{\rm d}x}{\sqrt{(1-x^2)(1-s^2x^2)}}=K(s).
\]
Then, observing that $A(0)=0$ we infer
\begin{equation}\label{ant}
A(s)=\int_0^sK(k){\rm d}k.
\end{equation}
To end our proof, we simply need to integrate term by term \eqref{ant} using the hypergeometric series representation of the first kind complete elliptic integral $K(k)$
\[
\begin{split}
A(s)&=\frac{\pi}{2}\int_{0}^{s}\, _2\mathrm{F}_{1}\left( 
\begin{array}{c}
1/2,1/2 \\ 
1
\end{array}
; k^2\right){\rm d}k=\frac{\pi}{2}\int_{0}^{s}\sum_{m=0}^{\infty }\frac{\left( 1/2\right) _{m}\left(
1/2\right) _{m}}{\left( 1\right) _{m}}\frac{k^{2m}}{m!}{\rm d}k\\
&=\frac{\pi}{2}\sum_{m=0}^{\infty }\frac{\left( 1/2\right) _{m}\left(
1/2\right) _{m}}{\left( 1\right) _{m}(2m+1)}\frac{s^{2m+1}}{m!}=\frac{\pi}{2}s\sum_{m=0}^{\infty }\frac{\left( 1/2\right) _{m}\left(
1/2\right) _{m}\left( 1/2\right) _{m}}{\left( 1\right) _{m}\left( 3/2\right) _{m}}\frac{s^{2m}}{m!}
\end{split}\]
where we used the identity
\begin{equation}\label{idee}
\frac{\left( 1/2\right) _{m}}{\left( 3/2\right) _{m}}=\frac{1}{1+2m}.
\end{equation}
\end{proof}

In this way we obtained an alternative proof of the Ramanujan's formula \eqref{prodigi}: in fact from \eqref{ant} and \cite{Byrd} entry 615.01 p. 274, we deduce that 
\begin{equation}\label{A(1)}
A(1)=\frac{\pi}{2}\,_{3}\mathrm{F}_{2}\left(  
\begin{array}{c}
1/2,1/2,1/2 \\[2mm]
1,3/2
\end{array}
;1\right)
=\int_0^1\frac{\arcsin x}{x\sqrt{1-x^2}}{\rm d}x=2G.
\end{equation} 
It is worth noting that the underlying integral representation for $G$
\[
G=\frac12\int_0^1\frac{\arcsin t}{t\sqrt{1-t^2}}{\rm d}t
\]
is also used in \cite[p. 161]{bradley}. Formula \eqref{prodigi} stems from \eqref{e1} taking the limit $s\to1^{-}$. Identity \eqref{prodigi} is also obtained, using a different technique, by Borwein et el. \cite{borwein} Section 3 when $s=0$.

Following the same method, we are able to find two other known identities in a simple way. Consider the parametric integral
\begin{equation}\label{cds}
C(s):=\int_0^1\frac{\log\left(1+\sqrt{1-s^2x^2}\right)}{\sqrt{1-x^2}}\,{\rm d}x.
\end{equation}
We have that
\begin{theorem}\label{zeero}
For any $|s|<1$
\begin{equation}\label{eics}
C(s)=\frac{\pi}{2}\,\log2-\frac{\pi}{16}\,s^2\,_{4}\mathrm{F}_{3}\left(  
\begin{array}{c}
1,1,3/2,3/2 \\[2mm]
2,2,2
\end{array}
;s^2\right).
\end{equation}
\end{theorem}
\begin{proof}
Again, we differentiate with respect to $s$. Of course, this is admissible, and after some computations, we arrive at the following:
\begin{equation}\label{eics1}
C'(s)=\frac{1}{s}\left(\int_0^1\frac{{\rm d}x}{\sqrt{1-x^2}}-\int_0^1\frac{{\rm d}x}{\sqrt{(1-x^2)(1-s^2x^2)}}\right)=\frac{1}{s}\left(\frac{\pi}{2}-K(s)\right).
\end{equation}
We integrate (\ref{eics1}) by series. First we fix the easiest constant of integration observing that 
\[
C(0)=\int_0^1\frac{\log2}{\sqrt{1-x^2}}\,{\rm d}x=\frac{\pi}{2}\,\log2.
\]
Then, representing as before the complete elliptic integral of the first kind in terms of Gauss hypergeometric series, we can write (\ref{eics1}) as
\begin{equation}\label{eics2}
C'(s)=-\frac{\pi}{2}\sum_{m=1}^\infty\frac{\left( 1/2\right) _{m}^2}{(m!)^2}s^{2m-1}.
\end{equation}
Thus, integrating (\ref{eics2}) in $[0,s],\,s<1$ we obtain
\begin{equation}\label{eics3}
C(s)=C(0)-\frac{\pi}{2}\sum_{m=1}^\infty\frac{\left( 1/2\right) _{m}^2}{(m!)^2}\frac{s^{2m}}{2m}.
\end{equation}
In order to reach the formula (\ref{eics}) we use, of course, the former computation of $C(0)$ and we change index in the series at the right-hand side of (\ref{eics3}) obtaining
\begin{equation}\label{eics4}
\sum_{m=1}^\infty\frac{\left( 1/2\right) _{m}^2}{(m!)^2}\frac{s^{2m}}{2m}=\frac{s^2}{2}
\sum_{n=0}^\infty\frac{\left( 1/2\right) _{n+1}^2}{((n+1)!)^2}\frac{s^{2n}}{(n+1)}.
\end{equation}
Now using the identity, see \cite[eq. (5) page 9]{carlson} 
\[
(a)_{n+1}=a\,(a+1)_n
\]
which, of course, for $a=1/2$ reads as
\[
\left(\frac12\right)_{n+1}=\frac12\left(\frac32\right)_{n}
\]
equation (\ref{eics4}) can be written as
\begin{equation}\label{eics5}
\sum_{m=1}^\infty\frac{\left( 1/2\right) _{m}^2}{(m!)^2}\frac{s^{2m}}{2m}=\frac{s^2}{8}
\sum_{n=0}^\infty\frac{\left( 3/2\right) _{n}^2}{(n!)^2(n+1)^3}s^{2n}
\end{equation}
and, recalling the identity, which is an immediate consequence of the definition of Pochhammer's symbol,
\begin{equation}\label{eics6}
n+1=\frac{(2)_n}{(1)_n}=\frac{(2)_n}{n!},
\end{equation}
formula (\ref{eics}) follows from (\ref{eics6}), (\ref{eics5}), (\ref{eics4}) and (\ref{eics3}).
\end{proof}
\begin{remark}
The particular value $C(1)$ is indeed related to Catalan's constant $G.$ In fact, we have 
\begin{equation}\label{nick}
C(1)=\int_0^1\frac{\log\left(1+\sqrt{1-x^2}\right)}{\sqrt{1-x^2}}\,{\rm d}x=2G-\frac{\pi}{2}\,\log2.
\end{equation}
 We can achieve thus by changing the variable $x=\sin t$, so that
\begin{equation}\label{nick1}
C(1)=\int_0^{\frac{\pi}{2}}\log\left(1+\cos t\right)\,{\rm d}t
\end{equation}
and finally using equation (12) of \cite{Nick}. Notice that \eqref{nick} can be evaluated by Mathematica$_{\text{\textregistered}}.$ 
\end{remark}
\begin{remark}\label{duecin}
Since it is clear that from its definition, the parametric integral (\ref{eics3}) converges for $s\to1^{-}$; recalling (\ref{nick}) we have thus shown that
\begin{equation}\label{4effe3}
_{4}\mathrm{F}_{3}\left(  
\begin{array}{c}
1,1,3/2,3/2 \\[2mm]
2,2,2
\end{array}
;1\right)=16\log2-\frac{32}{\pi}G.
\end{equation}
Formula \eqref{4effe3} is a particular case, $n = 1$ therein, of a $_4F_3$ formula given at \cite[p. 819]{adamchik}, here we obtained it using a different method.
\end{remark}

Now, consider the parametric integral:
\begin{equation}\label{dds}
D(s):=\int_0^1\frac{\log\left(1-\sqrt{1-s^2x^2}\right)}{\sqrt{1-x^2}}\,{\rm d}x
\end{equation}
which is closely related to integral (\ref{cds}). Now, since
\begin{equation}\label{summa}
C(s)+D(s)=2\left(\int_0^1\frac{\log s}{\sqrt{1-x^2}}\,{\rm d}x-\int_0^1\frac{\log x}{\sqrt{1-x^2}}\,{\rm d}x\right)=\pi\log\frac{s}{2}
\end{equation}
integral $D(s)$ is also is evaluated in hypergeometric terms. As a matter of fact the same argument used to obtain (\ref{eics1}) leads to:
\begin{equation}\label{eids1}
D'(s)=\frac{1}{s}\left(\frac{\pi}{2}+K(s)\right).
\end{equation}
Formula (\ref{eids1}) enables us to obtain directly the hypergeometric representation of $D(s)$ without using (\ref{summa}).

\section{Applications of Theorem \ref{zeero}}
Theorem \ref{zeero} has three interesting consequences, that were pointed out by the anonymous Referee, 
to which we are grateful.  

\subsection{Relation to the Ramanujan-like series for $1/\pi$}

It turns out that the evaluation of the series presented in Remark \ref{duecin}, equation \eqref{4effe3} 
allows us to calculate in a more easy way the hypergeometric $_4{\rm F}_3(1)$ series 
\begin{equation}\label{campbell1}
 _4{\rm F}_3\left( 
\begin{array}{c}
1/2, 1/2, 1, 1 \\ 
2, 2, 2
\end{array}
; 1\right)=\sum_{n=0}^\infty\frac{\binom{2n}{n}^2}{16^n(n+1)^3}=16\log 2+\frac{48}{\pi}-32\, \frac{G}{\pi}-16.
\end{equation}
Series \eqref{campbell1} has recently found many applications in the evaluation of Ramanujan-like series for $1/\pi$ involving harmonic numbers as presented in the recent articles: \cite[Theorem 4.1]{campbell}, \cite[p. 636]{interplay2} and \cite[Theorem 5.12]{interplay}. 
In all the three papers the computation of \eqref{campbell1} is due to very sophisticated methods, while using our identity \eqref{4effe3} and some related summations we derive \eqref{campbell1} in a few steps. Notice that, as \cite{campbell} also remarks, Mathematica$_\text{\textregistered}$ is unable to evaluate the series \eqref{campbell1}.

We therefore provide our calculation of \eqref{campbell1}. Rewriting \eqref{4effe3} using the Pochhammer symbol, and then simplifying, using the gamma function for the terms $(3/2)_n$ and relation \eqref{eics6} to express $(2)_n$, we observe that \eqref{4effe3} can be writen as:
\begin{equation}\label{4effe3a}
16\log2-\frac{32}{\pi}G=\sum_{n=0}^\infty\frac{4}{\pi}\frac{\left(\Gamma(\frac32+n)\right)^2}{(n!)^2(n+1)^3}.
\end{equation}

Now we insert in \eqref{4effe3a} the nice identity connecting $\Gamma(\frac32+n)$ to the central binomial coefficient:
\[
\Gamma\left(\frac32+n\right)=\frac{(2n+1)\,n!\,\sqrt{\pi}}{2^{2n+1}}\,\binom{2n}{n}
\]
allowing to rewrite \eqref{4effe3a} and, so, \eqref{4effe3} as:
\begin{equation}\label{4effe3b}
16\log2-\frac{32}{\pi}G=\sum_{n=0}^\infty\frac{(2n+1)^2}{16^n(n+1)^3}\,\binom{2n}{n}^2.
\end{equation}
Formula \eqref{4effe3b} indicates the connection with \eqref{campbell1} simply by expanding $(2n+1)^2$:
\begin{equation}\label{4effe3c}
16\log2-\frac{32}{\pi}G=4\sum_{n=0}^\infty\frac{n}{16^n(n+1)^2}\,\binom{2n}{n}^2+\sum_{n=0}^\infty\frac{\binom{2n}{n}^2}{16^n(n+1)^3}.
\end{equation}
Henceforth formula \eqref{campbell1} follows form \eqref{4effe3c} showing that
\begin{equation}\label{4effe3d}
\sum_{n=0}^\infty\frac{4n}{16^n(n+1)^2}\,\binom{2n}{n}^2=16-\frac{48}{\pi }.
\end{equation}
Through the use of partial fraction decomposition we can write the left-hand side of \eqref{4effe3d} as
\begin{equation}\label{4effe3e}
\sum_{n=0}^\infty\frac{4}{(n+1)16^n}\,\binom{2n}{n}^2  -\sum_{n=0}^\infty\frac{4}{(n+1)^216^n}\,\binom{2n}{n}^2.
\end{equation}
To evaluate the first series in \eqref{4effe3e} we go backward expressing it as using the Gauss hypergeometric $_2{\rm F}_1.$ The first step is again in the representation of the central binomial coefficient
\[
\binom{2n}{n}=\frac{2^{2n}\,\left(\frac32\right)_n}{(2n+1)n!}=\frac{2^{2n}\,\left(\frac12\right)_n}{n!}.
\]
Notice that we used relation \eqref{idee}. 
Now use again \eqref{eics6} in order to express $(2)_n$ in terms of $n+1,$ thus we arrive at the evaluation
\[
\sum_{n=0}^\infty\frac{4}{(n+1)16^n}\,\binom{2n}{n}^2=4\sum_{n=0}^{\infty}\frac{\left(\frac12\right)_n^2}{(2)_n}\,\frac{1}{n!}=4\,_2{\rm F}_1\left( 
\begin{array}{c}
1/2, 1/2 \\ 
2
\end{array}
; 1\right)=4\,\frac{\Gamma(2)\Gamma(2-\frac12-\frac12)}{\Gamma(2-\frac12)\Gamma(2-\frac12)}=\frac{16}{\pi}.
\]
Notice that in the last step we employed Gauss theorem to evaluate $\,_2{\rm F}_1(1).$ The second series in \eqref{4effe3e} is evaluated in \cite[p.653]{interplay2} as
\[
\frac{64}{\pi}-16.
\]
Thus \eqref{4effe3d} holds true and this implies the statement \eqref{campbell1}. We observe that the second series in \eqref{4effe3e} can also be evaluated by an argument similar to the first.

\subsection{Family of hypergeometric identities}

Identity \eqref{eics} of Theorem \ref{zeero} is a starting point for a large family of hypergeometric identities. The idea is quite simple: multiply both sides of \eqref{eics} in Theorem \ref{zeero} by  a given function and then integrate over the interval $[0, 1].$ It is worth noting that Mathematica$_{\text{\textregistered}}$ is not able to recognize in closed form \eqref{pow1}, \eqref{pow2}, \eqref{pow3} below.

\begin{theorem}\label{zeerokon}

The following summation formulas hold true:
\begin{equation}\label{pow1}
_{4}\mathrm{F}_{3}\left(  
\begin{array}{c}
1,1,3/2,3/2 \\[2mm]
2,2,3
\end{array}
;1\right)=8+32 \log 2-64\,\frac{G}{\pi }-\frac{32}{\pi },
\end{equation}

\begin{equation}\label{pow2}
_{4}\mathrm{F}_{3}\left(  
\begin{array}{c}
1,1,3/2,3/2 \\[2mm]
2,3,3
\end{array}
;1\right)=96+64 \log 2-128\,\frac{G}{\pi }-\frac{320}{\pi },
\end{equation}

\begin{equation}\label{pow3}
_{4}\mathrm{F}_{3}\left(  
\begin{array}{c}
1,1,3/2,3/2 \\[2mm]
2,2,4
\end{array}
;1\right)=18+48 \log 2-96\,\frac{G}{\pi }-\frac{208}{3\,\pi }.
\end{equation}

\end{theorem}
\begin{proof}
We start with \eqref{pow1}. Multiplying by $s$ both sides of \eqref{eics} and integrating for $s\in[0,1]$, at right-hand side we get, after changing variable $s^2=\sigma$:
\begin{align}\nonumber
&\frac{\pi}{4}\,\log 2-\frac{\pi}{16}\,\int_0^1 s^3\,_{4}\mathrm{F}_{3}\left(  
\begin{array}{c}
1,1,3/2,3/2 \\[2mm]
2,2,2
\end{array}
;s^2\right)\,{\rm d}s\\
=&\frac{\pi}{4}\,\log 2-\frac{\pi}{32}\,\int_0^1 \sigma\,_{4}\mathrm{F}_{3}\left(  
\begin{array}{c}
1,1,3/2,3/2 \\[2mm]
2,2,2
\end{array}
;\sigma\right)\,{\rm d}\sigma\nonumber
\\
=&\frac{\pi}{4}\,\log 2-\frac{\pi}{32}\,_{4}\mathrm{F}_{3}\left(  
\begin{array}{c}
1,1,3/2,3/2 \\[2mm]
2,2,3
\end{array}
;1\right).\label{rs23}
\end{align}
Now, integrating the left-hand side of \eqref{eics} in the relevant double integral thus obtained the order of integration can be switched, and the closed-form evaluation for the resultant integral requires some work, but is feasible in terms of elementary functions:
\begin{align}
&\int_0^1 s\left(\int_0^1\frac{\log\left(1+\sqrt{1-s^2x^2}\right)}{\sqrt{1-x^2}}\,{\rm d}x\right)\,{\rm d}s\nonumber\\
=&\int_0^1\frac{1}{\sqrt{1-x^2}}\left(\int_0^1s\,\log\left(1+\sqrt{1-s^2x^2}\,\right)\,{\rm d}s\,\right)\,{\rm d}x\nonumber\\
=&\int_0^1\frac{1}{\sqrt{1-x^2}}\left(\frac{2+x^2-2 \sqrt{1-x^2}}{4 x^2}+\frac{1}{2} \log \left(1+\sqrt{1-x^2}\right)\right)\,{\rm d}x\nonumber\\
=&\left[\frac{1}{4} \left(\frac{2-2 \sqrt{1-x^2}}{x}-\arcsin x \right)\right]_0^1+\frac12\,\int_0^1\frac{\log \left(1+\sqrt{1-x^2}\right)}{\sqrt{1-x^2}}\,{\rm d}x\nonumber\\
=&\frac12-\frac\pi8+G-\frac{\pi}{4} \,  \log 2.\label{ls23}
\end{align}
But \eqref{rs23} and \eqref{ls23} represent the same real number, hence \eqref{pow1} follows. 

For \eqref{pow2} and \eqref{pow3} the idea is the same: to get \eqref{pow2} multiplying both sides of \eqref{eics} by $s\ln s$ the left-hand side of \eqref{eics} provides, exchanging the order of integration:
\begin{align}
&\int_0^1 s\log s\left(\int_0^1\frac{\log\left(1+\sqrt{1-s^2x^2}\right)}{\sqrt{1-x^2}}\,{\rm d}x\right)\,{\rm d}s\nonumber\\
=&\int_0^1\left(\frac{x^2+3 \sqrt{1-x^2}-3+2 \log 2}{4 x^2 \sqrt{1-x^2}}-\frac{\left(x^2+2\right) \log \left(\sqrt{1-x^2}+1\right)}{4 x^2 \sqrt{1-x^2}}\right)\,{\rm d}x\nonumber\\
=&\frac{3 \pi }{8}+\frac{\pi}{8}\,\log2-\frac{G}{2}-\frac{5}{4}\label{ls23b}.
\end{align}

At right-hand side of \eqref{eics}, after multiplication by $s\ln s$ we obtain, after changing the variable $s^2=\sigma$ and then integrating by part, looking for the primitive of the hypergeometric function:
\begin{align}\nonumber
&-\frac{1}{8} \pi  \log 2-\frac{\pi}{16}\,\int_0^1 s^3\,\log s\,_{4}\mathrm{F}_{3}\left(  
\begin{array}{c}
1,1,3/2,3/2 \\[2mm]
2,2,2
\end{array}
;s^2\right)\,{\rm d}s\\
=&-\frac{1}{8} \pi  \log 2-\frac{\pi}{64}\,\int_0^1 \sigma\,\log\sigma \,_{4}\mathrm{F}_{3}\left(  
\begin{array}{c}
1,1,3/2,3/2 \\[2mm]
2,2,2
\end{array}
;\sigma\right)\,{\rm d}\sigma\nonumber
\\
=&-\frac{1}{8} \pi  \log 2-\frac{\pi}{256}\,_{4}\mathrm{F}_{3}\left(  
\begin{array}{c}
1,1,3/2,3/2 \\[2mm]
2,3,3
\end{array}
;1\right).\label{rs23b}
\end{align}
Notice that, integrating by series we used the relation $(3)_n=\dfrac{n+2}{2}(2)_n\,.$ Identity \eqref{pow2} then follows equating \eqref{ls23b} and \eqref{rs23b}.

To prove \eqref{pow3} we can multiply by $s^3$ both sides of  \eqref{eics} and repeat the argument, using, when integrating the hypergeometric series, the Pochhammer identity
\[
(4)_n=\frac{n+3}{3}(3)_n=\frac{n+3}{3}\frac{n+2}{2}\,\,(2)_n.
\] 
\end{proof}

\section{Catalan's constant from the complete elliptic integral of the second kind}\label{tre}
The parametric integral leading to the elliptic integral of the second kind is a slightly more intricate; nevertheless, we have:
\begin{theorem}
The following formula for $|s|<1$ holds:
\begin{equation}\label{e2}
\int_0^1\frac{\arcsin(sx)+sx\sqrt{1-s^2x^2}}{2x\sqrt{1-x^2}}{\rm d}x=\frac{\pi}{2}\,s\,_{3}\mathrm{F}_{2}\left(  
\begin{array}{c}
-1/2,1/2,1/2 \\[2mm]
1,3/2
\end{array}
;s^2\right).
\end{equation}
\end{theorem}
\begin{remark}
Formula \eqref{e2}, unlike many of the previous formulas given in this article, has not previously appeared in the literature.
\end{remark}

\begin{proof}
This time we put
\[
B(s)=\int_0^1\frac{\arcsin(sx)+sx\sqrt{1-s^2x^2}}{2x\sqrt{1-x^2}}{\rm d}x
\]
in such a way
\[
B'(s)=\int_0^1\sqrt{\frac{1-s^2x^2}{1-x^2}}\,{\rm d}x=E(s).
\]
Then recalling that $F(0)=0$, we arrive at
\begin{align}\label{ant2}
B(s)=\int_0^sE(k){\rm d}k=\frac{\pi}{2}\int_{0}^{s}\sum_{m=0}^{\infty }\frac{\left(- 1/2\right) _{m}\left(
1/2\right) _{m}}{\left( 1\right) _{m}}\frac{k^{2m}}{m!}{\rm d}k,
\end{align}
where we used the hypergeometric representation of $E(k)$ \eqref{repre}.
Our statement then follows from the integration by series as in Theorem \ref{uno}.
\end{proof}
\begin{remark}
The particular $_{3}\mathrm{F}_{2}$ specification and its relations to $G$ and the complete ellpitic integral of second kind in the following formula \eqref{3effeddue}, was also considered in \cite[p. 7 eq. 20]{adamchik}.
\end{remark}

\begin{corollary}
\begin{equation}\label{3effeddue}
\frac{1}{2}+G=\frac{\pi}{2}\,_{3}\mathrm{F}_{2}\left(  
\begin{array}{c}
-1/2,1/2,1/2 \\[2mm]
1,3/2
\end{array}
;1\right).
\end{equation}
\end{corollary}

\begin{proof}
From \eqref{ant2} and \cite[entry 615.01 p. 274]{Byrd}, we know that 
\begin{equation}
B(1)=\int_0^1\frac{\arcsin(x)+x\sqrt{1-x^2}}{2x\sqrt{1-x^2}}{\rm d}x=\frac12+G.
\end{equation} 
Formula (\ref{3effeddue}) stems from (\ref{e2}) taking the limit $s\to1^{-}$.
\end{proof}

\section*{Conclusion}

We have investigated the relationships between Catalan's constant and the $_3{\rm F}_2$ and $_4{\rm F}_3$ hypergeometric functions, equations (\ref{A(1)}), (\ref{4effe3}) and (\ref{3effeddue}), deriving them from some parametric integrals. We believe that our approach can be used to find further identities involving other math constants. Our approach is focused on investigation of the hypergeometric nature of $G$, following what was
first highlighted by Ramanujan, rather than following the example of many recent contributions and searching for particular numerical series increasing the speed of approximation of $G$.

\end{document}